\documentclass[11pt,a4paper,reqno,english]{amsart}
\title[]{%
  Dynamical Amrein--Berthier Uncertainty for
  Fractional Schr\"odinger Flows
}

\usepackage{%
    amssymb%
    ,babel%
    ,csquotes%
    ,graphicx%
    ,mathrsfs%
    ,eucal%
    ,xcolor%
    ,enumitem%
    ,geometry%
    ,esint}
\usepackage[%
    texencoding=ascii%
    ,backend=biber%
    ,url=false%
    ,doi=false%
    ,isbn=false%
    ,giveninits=true%
    ,maxnames=5%
    ,hyperref]{biblatex}
\usepackage[%
    pdfauthor={Piero D'Ancona}%
    ,colorlinks=true%
    ,linkcolor=magenta%
    ,citecolor=cyan%
    ,urlcolor=blue%
    ,hyperfootnotes=false%
]{hyperref}

\DeclareMathOperator{\spt}{spt}

\newcommand{\bra}[1]{\langle #1 \rangle}
\newcommand{\one}[1]{\mathbf{1}_{#1}}

\numberwithin{equation}{section}
\newtheorem{theorem}{Theorem}[section]
\newtheorem{corollary}[theorem]{Corollary}
\newtheorem{lemma}[theorem]{Lemma}
\newtheorem{proposition}[theorem]{Proposition}
\theoremstyle{remark}
\newtheorem{remark}[theorem]{Remark}

\theoremstyle{definition}

\date{\today}

\author[P.~D'Ancona]{Piero D'Ancona}
\address{Piero D'Ancona:
Dipartimento di Matematica\\ Sapienza Universit\`{a} di Roma\\
Piazzale A.~Moro 2\\ 00185 Roma\\ Italy}
\email{piero.dancona@uniroma1.it}
\author[D.~Fiorletta]{Diego Fiorletta}
\address{Diego Fiorletta:
Dipartimento di Matematica\\ Sapienza Universit\`{a} di Roma\\
Piazzale A.~Moro 2\\ 00185 Roma\\ Italy}
\email{diego.fiorletta@uniroma1.it}

\thanks{%
The authors are partially supported by the MIUR PRIN project
2020XB3EFL, ``Hamiltonian and Dispersive PDEs'', by the Progetto
Ricerca Scientifica 2024 ``Wave dynamics in heterogeneous media''
of Sapienza University,and
by the Gruppo Nazionale per l'Analisi Matematica, la
Probabilit\`{a} e le loro Applicazioni (GNAMPA),
Project CUP E53C23001670001.
The second author is partially supported by the Progetto Avvio alla
Ricerca 2025 ``Indeterminazione dinamica e sue applicazioni'' of Sapienza Universit\`{a} di Roma.
}

\begin{document}

\begin{abstract}
  We prove dynamical Amrein--Berthier uncertainty principles for
  fractional Schr\"odinger flows. For the free Hamiltonian
  $H=(-\Delta)^\alpha$ on $L^2(\mathbb{R}^n)$, with
  $\alpha>\frac{1}{2}$, we show that two--time localization on
  finite measure sets $E,F$ forces the quantitative estimate
  \begin{equation*}
    \|u(t)\|_{L^{2}}\lesssim_{E,F,T,n,\alpha}
    \|u(0)\|_{L^{2}(E^{c})}
    +
    \|u(T)\|_{L^{2}(F^{c})},
    \qquad
    T\neq0,\ t\in \mathbb{R}
  \end{equation*}
  for $u(t)=e^{-itH}u(0)$
  at every time. The threshold $\alpha>\frac{1}{2}$ is tied to
  the stationary phase structure of the fractional kernel. If
  $\alpha\ge1$ the sets can be arbitrary finite measure sets; if
  $\frac{1}{2}<\alpha<1$ we impose the finiteness of a natural
  interaction energy 
  \begin{equation*}
    \textstyle
    \mathcal{I}_{\gamma}(E,F)
    =
    \int_{\mathbb{R}^n \times \mathbb{R}^n}
    \one{F}(x)|x-y|^{2\gamma}\one{E}(y)\,dx\,dy<\infty,
    \qquad
    \gamma = \frac{n(1-\alpha)}{2 \alpha-1}
  \end{equation*}
  of the pair $(E,F)$, essentially equivalent to a
  sufficiently fast joint decay of the measure of the sets at infinity. In particular,
  compact support at two distinct times is impossible for a
  nonzero solution. We also prove corresponding results for
  one dimensional fractional Hamiltonians
  $(-\partial_x^2+V)^\alpha$ under weighted scattering
  assumptions, and for higher order Hamiltonians $(-\Delta)^m+V$
  for suitable classes of decaying potentials $V$.
\end{abstract}

\maketitle



\section{Introduction}
\label{sec:intro}

Uncertainty principles express the impossibility of localizing a
function too strongly in two different representations, typically
in physical and Fourier variables; see e.g.~the survey
\cite{FollandSitaram97-a}. The classical forms, from Heisenberg
to Hardy \cite{Hardy33-a} and Beurling--H\"{o}rmander
\cite{Hormander91-a}, measure localization through moments,
Gaussian decay, or exponential weights. Another line of results
uses localization on sets of finite measure: the theorem of
Amrein and
Berthier \cite{AmreinBerthier77-a} says that a nonzero
$L^2$ function and its Fourier transform cannot both be
supported on sets of finite measure. Benedicks gave the
corresponding qualitative form \cite{Benedicks85-a}, while
Nazarov and Jaming developed quantitative refinements
\cite{Nazarov93-a,Jaming07-a}. This is the circle of results to
which the present paper belongs.

A \emph{dynamical} Amrein--Berthier principle replaces the
Fourier transform by a time evolution. Given a selfadjoint
Hamiltonian $H$ and the propagator $U(T)=e^{-iTH}$, one asks
whether a solution can be localized in a finite measure set $E$
at time $0$ and in another such set $F$ at time $T\neq0$.
For the free Schr\"odinger Hamiltonian $-\Delta$ this is
essentially the
classical theorem again, because the propagator at nonzero time
is a Fourier transform up to elementary unitary factors. Related
dynamical forms of Hardy's uncertainty principle for
Schr\"odinger evolutions were developed by
Escauriaza, Kenig, Ponce and Vega
\cite{EscauriazaKenigPonce08-a,EscauriazaKenigPonce10-a}.
In our previous work 
\cite{DAnconaFiorletta26-a} we proved
dynamical versions of the Amrein--Berthier principle
for Schr\"odinger operators with subquadratic
potentials. There the main point was to show that suitable
localized propagators are compact; once this is known, unique
continuation yields the
two--time uncertainty estimate.

The present paper asks the same question for fractional
Schr\"odinger flows. For $\alpha>0$, the free fractional flow is
the unitary group generated by
$H=(-\Delta)^\alpha$ on $L^2(\mathbb{R}^n)$:
\begin{equation*}
  u(t)=e^{-it(-\Delta)^\alpha}u_0,
  \qquad
  \widehat{u}(t,\xi)
  =
  e^{-it|\xi|^{2\alpha}}\widehat{u_0}(\xi).
\end{equation*}
Equivalently, the propagator is the convolution with the inverse
Fourier transform of $e^{-it|\xi|^{2\alpha}}$. This is a
nonlocal evolution unless $\alpha$ is an integer, and the
kernel is no longer the explicit quadratic Schr\"odinger kernel
except when $\alpha=1$. Thus the finite-measure uncertainty
problem is controlled by a genuinely fractional oscillatory
kernel.

This nonlocality introduces a new feature already in the free
case. The phase $|\xi|^{2\alpha}$ has order larger than one
exactly when $\alpha>1/2$; below or at the endpoint one meets
the half--wave type conic behavior. In the range
$1/2<\alpha<1$, stationary phase gives the sharp 
growth
\begin{equation*}
  \left|
    \mathcal{F}(e^{i|\cdot|^{2\alpha}})(x)
  \right|
  \simeq
  |x|^{\frac{n(1-\alpha)}{2\alpha-1}}
  \qquad
  \text{at the level of leading asymptotics.}
\end{equation*}
For $\alpha\ge1$ the same kernel is bounded at infinity. The
main theorem reflects precisely this dichotomy. If
$\alpha\ge1$, arbitrary finite measure sets $E,F$ give
compactness of
$\one{F}e^{-iT(-\Delta)^\alpha}\one{E}$. If
$1/2<\alpha<1$, finite measure alone is not the natural
hypothesis, and one must pay the polynomial
growth of the kernel through the interaction energy
$\mathcal{I}_{\gamma}(E,F)$ introduced below.

After compactness of the localized propagator
$L=\one{F}U\one{E}$ is proved, the passage to the
Amrein--Berthier inequality is abstract. 
If one can exclude that $L$ has operator norm 1,
then the full $L^2$ norm of the flow is controlled by the 
mass outside
$E$ at time $0$ and outside $F$ at time $T$. What is new here is
not this abstract step, but the analytic input that provides
compactness in a nonlocal fractional setting.

The paper contains two more
perturbative extensions. In dimension one we
consider fractional powers of $-\partial_x^2+V$ on the positive
absolutely continuous subspace. The proof uses the Jost
representation of the resolvent, classical one dimensional
scattering theory \cite{DeiftTrubowitz79-a}, and Fourier
estimates in the spirit of
\cite{DAnconaFanelli06-a,Mizutani11-a}.
The behavior of the Wronskian at zero energy forces the usual
distinction between regular and resonant thresholds. Finally, we
record a higher order extension for $(-\Delta)^m+V$, $m\ge2$,
in the range where endpoint bounds for the wave operators are
available, relying on the work of Erdogan and Green
\cite{ErdoganGreen22-a,ErdoganGreen23-a}. Recent results
on fractional order perturbations
\cite{ErdoganGoldbergGreen25-a,ErdoganGoldbergGreen25-b}
suggest that the perturbative fractional part of the theory
should have further room to grow.

The estimates also have a standard control theoretic
consequence: by the Hilbert uniqueness method they imply an
observability inequality, and hence exact controllability from
the complement of the observation set (see
e.g.~\cite{Zuazua03-a}). This as an interesting
directon and will be further explored elsewhere.

One core step in the proof can be packaged in
an abstract criterion. If one proves that the
localized propagator $\one{F}U\one{E}$ has operator norm
\emph{strictly} less than one, the uncertainty inequality
follows immediately:

\begin{proposition}[Compactness of the localized dynamics implies
uncertainty] \label{prop:compactCriterion}
  Let $U$ be a unitary operator on $L^2(\mathbb{R}^n)$ and let
  $E,F\subset\mathbb{R}^n$ be measurable sets. If
  \begin{equation*}
    \delta:=\|\one{F}U\one{E}\|_{L^2\to L^2}<1,
  \end{equation*}
  then there is a constant $C=C(\delta)$ such that
  \begin{equation}
    \label{eq:abstractAB}
    \|v\|_{L^2}
    \leq C\left(
    \|v\|_{L^2(E^c)}+\|Uv\|_{L^2(F^c)}
    \right)
  \end{equation}
  for all $v\in L^2(\mathbb{R}^n)$. In particular, if
  $\one{F}U\one{E}$ is compact and there is no nonzero $v\in L^2$
  such that $\operatorname{spt}v\subset E$ and
  $\operatorname{spt}Uv\subset F$, then
  \eqref{eq:abstractAB} holds.
\end{proposition}

\begin{proof}
  Write $P_E=\one{E}$, $Q_E=\one{E^c}$, $P_F=\one{F}$
  and $Q_F=\one{F^c}$. Since
  \begin{equation*}
    P_Ev=P_EU^{-1}P_FUP_Ev+P_EU^{-1}Q_FUP_Ev,
  \end{equation*}
  we have
  \begin{equation*}
    \|P_Ev\|_2
    \leq
    \delta\|P_Ev\|_2+\|Q_FUv\|_2+\|UQ_Ev\|_2.
  \end{equation*}
  Thus
  \begin{equation*}
    (1-\delta)\|P_Ev\|_2
    \leq
    \|Q_FUv\|_2+\|Q_Ev\|_2,
  \end{equation*}
  and \eqref{eq:abstractAB} follows from
  $\|v\|_2\leq\|P_Ev\|_2+\|Q_Ev\|_2$. For the final assertion,
  suppose that $\one{F}U\one{E}$ is compact and has norm one. Its
  norm is then attained by some $v$ with $\|v\|_2=1$. Equality in
  the two orthogonal projections and in the unitary factor forces
  $P_Ev=v$ and $P_FUv=Uv$, which is precisely the
  excluded two--time localization.
\end{proof}


We give a brief synopsis of the paper. Section
\ref{sec:intr} fixes notations and describes our main results. 
Section \ref{sec:free_frac_flow} proves the oscillatory estimate 
for the free fractional kernel and derives compactness of 
localized free propagators. Section \ref{sec:one_dim_flow} treats
fractional powers of one dimensional perturbed Schr\"odinger
operators through the Jost representation and the corresponding
Fourier estimates. The final section combines these compactness
results with the abstract criterion above to prove the main
uncertainty inequalities.


\textbf{Data Availability Statement}. No new data were created
or analysed in this study. Data sharing is not applicable
to this article.

\textbf{Conflict Of Interest Statement}. All authors declare
that they have no conflicts of interest.

\section{Main results}
\label{sec:intr}

Given a real number $\alpha>0$, the fractional Laplacian
$(-\Delta)^\alpha$ is the nonnegative selfadjoint operator on
$L^2(\mathbb{R}^n)$ with Fourier multiplier $|\xi|^{2\alpha}$.

For $\gamma\ge0$ and measurable sets $E,F\subset\mathbb{R}^n$ we
define the interaction energy
\begin{equation}
  \label{eq:interactionEnergy}
  \mathcal{I}_{\gamma}(E,F)
  =
  \int_{\mathbb{R}^n \times \mathbb{R}^n}
  \one{F}(x)|x-y|^{2\gamma}\one{E}(y)\,dx\,dy.
\end{equation}
Thus $\mathcal{I}_{0}(E,F)=|E||F|$, while positive $\gamma$
measures how far the two sets interact at spatial infinity.

Our first result is the following.

\begin{theorem}[Fractional Amrein--Berthier Inequality]
\label{the:FracDynAB}
  Let $n\ge 1$, $\alpha > \frac{1}{2}$ and $H = (-
  \Delta)^{\alpha}$. Let $E,F \subset \mathbb{R}^{n}$ be two
  sets of finite measure, and if $\alpha<1$ further assume that
\begin{equation}
  \label{eq:MeasDecay2}
  \mathcal{I}_{\gamma}(E,F)<+\infty,
  \qquad
  \gamma = \frac{n(1-\alpha)}{2 \alpha-1}.
\end{equation}
  Then for any $T\neq0$ there is a constant $C=C(E,F,T,n,\alpha)$
  such that any solution to $i \partial_{t}u=Hu$ satisfies
\begin{equation}
  \label{eq:AmBerIneq}
  \|u(t)\|_{L^{2}}\le
  C (  \|u(0)\|_{L^{2}(E^{c})}+\|u(T)\|_{L^{2}(F^{c})} )
  \qquad
  \forall t\in \mathbb{R}.
\end{equation}
\end{theorem}


In particular we have the following result
(compare with \cite{HuangSoffer21-a} and \cite{ChoiWalton26-a}):

\begin{corollary}[Compact support at two times]
\label{cor:CompactSupportFrac}
  Let $H=(-\Delta)^\alpha$ with $\alpha>\frac{1}{2}$ and let
  $T\neq0$. If $u(t)=e^{-itH}u_0$ satisfies $\operatorname{spt}
  u(0)\subset E$ and $\operatorname{spt} u(T)\subset F$ for two
  compact sets $E,F\subset\mathbb{R}^n$, then $u\equiv0$.
\end{corollary}

\begin{proof}
  Compact sets satisfy the hypotheses of Theorem
  \ref{the:FracDynAB}, and the right hand side of
  \eqref{eq:AmBerIneq} is zero. 
\end{proof}

\begin{remark}[The threshold $\alpha>\frac{1}{2}$]
\label{rem:thresholdAlpha}
The restriction $\alpha>\frac{1}{2}$ is exactly the condition
that the oscillatory phase $|\xi|^{2\alpha}$ have exponent
larger than one. In the proof of Proposition
\ref{pro:estoscilint}, this is what produces a genuine kernel
with polynomial bounds. More precisely, for
$\frac{1}{2}<\alpha<1$ the leading term of
$\mathcal{F}(e^{i|\cdot|^{2\alpha}})$ grows like
$|x|^{\gamma}$, where
\begin{equation*}
  \gamma=\frac{n(1-\alpha)}{2\alpha-1}.
\end{equation*}
The interaction condition \eqref{eq:MeasDecay2} is therefore
exactly the condition which makes the localized kernel
Hilbert--Schmidt within the present argument. At the endpoint
$\alpha=\frac{1}{2}$ the multiplier $e^{it|\xi|}$ has a
wave type conic singularity rather than the fractional
Schrodinger smoothing used here, and the Fourier transform is no
longer controlled by the Hilbert--Schmidt argument below. Thus
the present method can not reach the endpoint.
\end{remark}

The finiteness of $\mathcal{I}_{\gamma}(E,F)$ is a condition on
the joint size at infinity of the two sets. For instance, it
is implied by
\begin{equation}
  \label{eq:MeasDecayCond}
  \textstyle
  |(E\cup F) \setminus B_{R}(0)| \lesssim
  R^{-2 \gamma  -1-\epsilon},
  \qquad
  \gamma=\frac{n(1-\alpha)}{2 \alpha-1}.
\end{equation}
Indeed, one has
\begin{equation*}
  \mathcal{I}_{\gamma}(E,F)\lesssim
  \int_{E}  \bra{x}^{2 \gamma} dx
  \int_{F}\bra{y}^{2 \gamma}dy
\end{equation*}
and
\begin{equation*}
  \int_{E}\bra{x}^{2\gamma}dx\lesssim
  1+\sum_{k\ge2}\int_{B_{k}(0)\setminus B_{k-1}(0)}
  |x|^{2\gamma}\one{E}dx \lesssim
  1+\sum_{k\ge2}k^{2\gamma}k^{-2\gamma-1-\epsilon}<\infty.
\end{equation*}

Our methods allow us to treat one dimensional fractional
operators with potential terms. Define the following
weighted $L^1$ spaces,
\begin{equation}
  \label{eq:weightL1}
  \begin{aligned}
  L^1_{N}(\mathbb{R})
  &=
  \{
  f : \mathbb{R} \rightarrow \mathbb{R}
  \text{ s. t. }
  \lVert \bra{x}^N f \rVert_1 < + \infty
  \},
  \end{aligned}
\end{equation}
for $N\ge0$. Obviously, $L^1_{N} \subset L^1_{M}$ for
$N > M$.

Let $V \in L^1_1(\mathbb{R})$, real valued, and set
\begin{equation}
  \label{eq:onedham}
  h=-\partial_x^2+V .
\end{equation}
It is well known that
$h$ is selfadjoint in $L^2(\mathbb{R})$ with form
domain $H^1(\mathbb{R})$; its absolutely continuous spectrum is
the positive half-line $[0,+\infty)$, the singular continuous
spectrum is absent, and its eigenvalues, if present, are
strictly negative. If we assume $V\ge0$, the spectrum is
entirely absolutely continuous, the operator is nonnegative,
and we may define for $\alpha>0$ the fractional powers
$h^{\alpha}$
via the spectral theorem on the positive spectrum $[0,+\infty)$.

The Jost functions $f_{\pm}(\lambda)$ are solutions
to the equation:
\begin{equation}
  \label{eq:LipSchw}
  ( -\partial^2_x + V) f (\lambda , x) = \lambda^2 f(\lambda,x),
\end{equation}
which satisfy the asymptotic conditions
\begin{equation}
  \label{eq:jostAsymp}
  |f_{\pm}(\lambda,x) - e^{\pm i \lambda x}|
  \rightarrow 0  \qquad \text{ as } x \rightarrow \pm \infty.
\end{equation}
They exist and are unique assuming $V
\in L^1_1(\mathbb{R})$; see
\cite{DeiftTrubowitz79-a}. Their Wronskian
\begin{equation}
  \label{eq:wronsk}
  \begin{aligned}
  W(\lambda)
  &=
  f_{+}(\lambda,x) \cdot
  \partial_x f_{-}(\lambda,x)
  \\
  &\quad
  -
  f_{-}(\lambda,x) \cdot
  \partial_x f_{+}(\lambda,x),
  \end{aligned}
\end{equation}
is independent of $x$. For real nonzero $\lambda$ it does not
vanish, and the only possible threshold zero is at $\lambda=0$.

In the following, we will make use of two different assumptions
over the potential:

Assumption \textbf{(A)}. $\exists N \in \mathbb{N}$, $N >
\frac{4 \alpha - 1}{2 \alpha - 1}$ such that $V \in
L^1_N(\mathbb{R})$ and $W(0) \neq 0$.

Assumption \textbf{(B)}. $\exists N \in \mathbb{N}$, $N >
\frac{6 \alpha - 2}{2 \alpha -1}$ such that $V \in
L^1_N(\mathbb{R})$ and $ W(0) = 0$.

\begin{theorem}[Fractional A--B Inequality II]
\label{the:PotFracDynAB}
  Let $\alpha > \frac{1}{2}$, assume that $V\ge0$ satisfies
  Assumption \textbf{(A)} or \textbf{(B)} and let $H=h^{\alpha}$.
  Then, for
  any $T \neq 0$ and any compact sets $E,F\subset\mathbb{R}$,
  there is a constant $C=C(E,F,T,n,\alpha,V)$ such that every
  solution $u(t)=e^{-itH}u_0$, with $u_0\in
  L^2(\mathbb{R})$, satisfies \eqref{eq:AmBerIneq}.
\end{theorem}



By adapting our methods, we can also treat higher order
Schr\"odinger operators of the form
\begin{equation}
  \label{eq:powLapHam}
  H=(-\Delta)^m + V, \qquad 2 \leq m \in \mathbb{N},
\end{equation}
acting as unbounded operators on $L^2(\mathbb{R}^n)$. We will
make use of regularizing effects of free higher
order Schr\"odinger propagators \cite{Balabane89-a},
still valid for flows generated by the perturbed hamiltonian
\eqref{eq:powLapHam}, as long as the wave operators
\begin{equation}
  \label{eq:waveOp}
  W_{\pm}
  =
  s - \lim_{t \rightarrow \pm \infty}
  e^{i t H} e^{- i t (- \Delta)^m}
\end{equation}
enjoy the $L^{\infty}$ (endpoint) boundedness. Such property
was proved in
\cite{ErdoganGreen22-a,ErdoganGreen23-a},
in the range $n > 2m$, under suitable assumptions for $V$, that
we reproduce here for the benefit of the reader.

\textbf{Assumption (H)}: for the real-valued potential $V$,
\begin{enumerate}
\item $\exists \delta > 0$ such that
\begin{itemize}
\item $\lVert \bra{\cdot}^{\frac{4m + 1 - n}{2}+\delta} V(\cdot) \rVert_2 \lesssim 1$ when $2m < n <4m -1$;
\item  $\lVert \bra{\cdot}^{1+\delta} V(\cdot) \rVert_{H^{\delta}} \lesssim 1$ when $n=4m -1$;
\item $\lVert \mathcal{F} ( \bra{\cdot}^{\sigma} V(\cdot) ) \rVert_{L^{\frac{n-1-\delta}{n-2m-\delta}}}$, 
for some $\sigma > \frac{2n - 4m}{n-1-\delta} + \delta$, when $n > 4m -1$. 
\end{itemize}
\item $|V(x)| \lesssim \bra{x}^{-\beta}$, with $\beta > n + 5$ when $n$ is odd, and $\beta > n + 4$ when $\beta$ is even.
\item The associated higher order Schr\"odinger operator \eqref{eq:powLapHam} has no positive eigenvalues and zero energy is regular (this means that the only distributional
  solution to $H \psi=0$ with $\bra{x}^{\frac n2-2m-}\psi\in L^{2}$
  is 0).
\end{enumerate}

Notice that under our assumptions on the potential, the singular continuous spectrum of $H$ is empty, and its eigenvalues, if present, are finite and negative (see e.g.~\cite{EgorovKondratev90-a}).

\begin{theorem}[Higher-order extension]
\label{the:hiOrdAB}
Let $H$ as in \eqref{eq:powLapHam}. Suppose
\begin{equation}
  \label{eq:dimAssump}
  n \geq 7 \text{ if } m = 2, \qquad
  n > 2m \text{ if } m \geq 3.
\end{equation}
Moreover, suppose that $V$ satisfies Assumption \textbf{(H)}. Then, for
any $T \neq 0$ and any finite measure sets $E$,$F \subset
\mathbb{R}^n$, solutions to $i \partial_{t} u = H
u$ satisfy \eqref{eq:AmBerIneq}, for some constant $C = C(E,F,T,n,\alpha,V)$.
\end{theorem}

Uncertainty principles have many applications to other areas of
PDE theory, notably to control theory. 
As a simple application, we can deduce the following
observability inequality for the propagator $e^{-itH}$,
with $H=(-\Delta)^{\alpha}$ or $h^{\alpha}$ or
$(-\Delta)^{m}+V$:

\begin{theorem}
\label{the:Obs1}
  Under the assumptions of theorem \ref{the:FracDynAB} (resp.
  theorem \ref{the:PotFracDynAB}, theorem \ref{the:hiOrdAB}),
  for any $T>0$ and for any set $E$ for which
  \eqref{eq:AmBerIneq} holds, we have
\begin{equation}
  \label{eq:Obs2}
  \textstyle
  \|u_{0}\|^2_{L^2} \lesssim_{E,T,n,\alpha,V}
  \int_0^T \| e^{- i t H}u_{0}\|^2_{L^2(E^c)} dt
  \qquad\text{for all}\quad u_{0}\in L^{2}.
\end{equation}
\end{theorem}

Combining \eqref{eq:Obs2} with the
\emph{Hilbert uniqueness method}
\cite{Zuazua03-a}, one deduces by a standard
procedure the exact controllability property for $e^{-iTH}$.
This means that, given a set $E$ as above, for any $u_0,u_T \in
L^2(\mathbb{R}^n)$ we can always find $\nu \in
L^2(\mathbb{R}^n \times (0,T))$ such that the unique solution of
the Cauchy problem
\begin{equation}
  \label{eq:controleq}
  i \partial_t u = H u + \one{E^c} \nu,
  \qquad
  u(0)=u_{0}
\end{equation}
reaches the target state at time $T$ i.e. $u(T)=u_T$.


\section{The free fractional Schr\"odinger flow}
\label{sec:free_frac_flow}

Notice that when $H = (- \Delta)^{\alpha}$, the (backwards)
propagator $e^{i T H}$ at time $T \neq 0$ admits a natural
representation as an integral convolution operator,
\begin{equation}
  \label{eq:FracSchKer}
  e^{i T H} u_0
  =
  \mathcal{F} (e^{ i T |\cdot|^{2 \alpha}}) * u_0.
\end{equation}
We are thus led to study the Fourier transform of the
symbol $e^{i |\xi|^{a} }$.

\subsection{An oscillatory integral estimate}
\label{sub:a_four_tran}

\begin{proposition}[]
\label{pro:estoscilint}
  Let $n\ge1$, $a>1$ and $v(\xi)=e^{i|\xi|^{a}}$ on
  $\mathbb{R}^{n}$. Then $\widehat{v}$ is a function and
  \begin{equation*}
    |\widehat{v}(x)|\lesssim_{a,n}
    \begin{cases}
    1 &\text{if $ a\ge2 $,}\\
    \bra{x}^{\frac{n(2-a)}{2(a-1)}} &\text{if $ 1<a<2 $.}
    \end{cases}
  \end{equation*}
  In the range $1<a<2$ this power is sharp. More precisely, if
  \begin{equation*}
    \beta=\frac{n(2-a)}{2(a-1)},
    \qquad
    d_a=(a-1)a^{-\frac{a}{a-1}},
  \end{equation*}
  then there exists a constant $c_{a,n}\neq0$, depending also on
  the Fourier transform normalization, such that
  \begin{equation*}
    \widehat{v}(x)
    =
    c_{a,n}|x|^\beta e^{-id_a|x|^{\frac{a}{a-1}}}
    +
    o(|x|^\beta)
    \qquad
    \text{as } |x|\to+\infty.
  \end{equation*}
\end{proposition}

\begin{proof}
  Since $e^{(i-\epsilon)|\xi|^a}\to e^{i|\xi|^{a}}$ in
  $\mathscr{S}'$ for $\epsilon \downarrow0$, using the standard
  representation of the Fourier transform of radial functions,
  letting $\nu=\frac n2-1$ we have as $\epsilon \downarrow0$
\begin{equation*}
  \textstyle
  I_{\epsilon}(r)=
  r^{-\nu}
  \int_{0}^{\infty}
  J_{\nu}(r \rho)e^{(i-\epsilon)\rho^{a}}\rho^{\frac n2}
  d \rho=
  \int_{0}^{\infty}
  \widetilde{J}_{\nu}(r \rho)e^{(i-\epsilon)\rho^{a}}
  \rho^{n-1} d \rho
  \to
  (2\pi)^{\frac n2}
  \widehat{v}
  \quad\text{in}\quad \mathscr{S}'.
\end{equation*}
  If we prove that $\bra{r}^{-m}|I_{\epsilon}(r)|\le C$
  uniformly in $\epsilon$, by weak--$*$ compactness in
  $L^{\infty}$ we deduce that $\widehat{v}$ is a function and
  $\bra{x}^{-m}|\widehat{v}(x)|\le C$ as well.

  Splitting
  $I_{\epsilon}=\int_{0}^{1}+\int_{1}^{\infty}=:I^{1}+I^{2}$ and
  recalling that $\widetilde{J}_{\nu}$ is bounded with all its
  derivatives when $\nu\ge -\frac 12$, we have obviously
  $|I^{1}|\le C$ uniformly in $\epsilon$. Changing variables in
  $I^{2}$ we get
\begin{equation*}
  \textstyle
  I^{2}(r)=\frac 1a\int_{1}^{\infty}
  \widetilde{J}_{\nu}(r \rho^{\frac 1a})
  e^{(i-\epsilon)\rho}
  \rho^{\frac na-1} d \rho.
\end{equation*}
  Writing $e^{(i-\epsilon)\rho}=(i-\epsilon)^{-k}
  \partial_{\rho}^{k}e^{(i-\epsilon)\rho}$ and integrating by
  parts $k$ times we get
\begin{equation*}
  \textstyle
  =(i-\epsilon)^{-k}p_{k}(r)+
  \frac{1}{a(\epsilon-i)^{k}}
  \int_{1}^{\infty}e^{(i-\epsilon)\rho}
  \partial_{\rho}^{k}
  \left(\widetilde{J}_{\nu}(r \rho^{\frac 1a})
  \rho^{\frac na-1}\right)d \rho
\end{equation*}
  for some polynomial $p_{k}$ of degree $k$ with coefficients
  depending on $n,a$. Expanding the derivatives inside the
  integral, the term with minimal decay in $\rho$ occurs when
  $k$ derivatives hit $\widetilde{J}_{\nu}$, giving
\begin{equation*}
  \widetilde{J}_{\nu}^{(k)}(r \rho^{\frac 1a})
  \rho^{\frac na-1+k(\frac 1a-1)}r^{k}.
\end{equation*}
  If $k>\frac{n}{a-1}$ all terms are integrable and we conclude
\begin{equation*}
  \textstyle
  |I_{\epsilon}(r)|\lesssim_{n,a}\bra{r}^{k},
  \qquad
  k>\frac{n}{a-1}.
\end{equation*}
  Hence $\widehat{v}$ is a locally bounded function satisfying
  $|\widehat{v}(x)|\lesssim \bra{x}^{k}$, $k>\frac{n}{a-1}$.
  This bound is sufficient for $|x|\le1$.

  For large $x$ we use a different approximation: we introduce a
  cutoff $\chi_{A}(\rho)=\chi(\rho/A)$ equal to 1 near the
  origin and we consider the integral
\begin{equation*}
  \textstyle
  I_{A}(r)=
  \int_{0}^{\infty}
  \widetilde{J}_{\nu}(r \rho)e^{i\rho^{a}}
  \chi_{A}(\rho) \rho^{n-1} d \rho.
\end{equation*}
  As before, it is sufficient to prove that
  $\bra{r}^{-m}|I_{A}(r)|\le C$ uniformly in $A$ to conclude
  that $\bra{x}^{-m}|\widehat{v}(x)|\le C$. We split the integral
  $I_{A}=\int_{0}^{1}+\int_{1}^{\infty}$; obviously, the first
  piece is bounded uniformly in $r\ge1,A>0$, and we can focus on
  the second one. We recall that for $\nu\ge-\frac 12$
\begin{equation*}
  \textstyle
  \widetilde{J}_{\nu}(s)=s^{-\frac{n-1}2}
  \left[
  e^{is}a_{+}(s)
  +
  e^{-is}a_{-}(s))
  \right]
\end{equation*}
  where $a_{\pm}(s)$ are bounded smooth functions
  on $s>0$ satisfying
\begin{equation}
  \label{eq:asyapm}
  \textstyle
  |a_{\pm}^{(k)}(s)|\lesssim_{\nu,k}s^{-k}
  \quad\text{for}\quad s\ge1,\ k\in \mathbb{N}_{0}.
\end{equation}
  Thus we can write
\begin{equation}
  \label{eq:decomp}
  \textstyle
  I_{A}(r)=\int_{0}^{1}+\int_{1}^{\infty}=
  \int_{0}^{1}+ r^{-\frac{n-1}{2}}(I_{+,A}+I_{-,A})
\end{equation}
  where
\begin{equation*}
  \textstyle
  I_{\pm,A}(r)=\int_{1}^{\infty}e^{i (\rho^{a}\pm r \rho)}
  a_{\pm}(r\rho)\rho^{\frac{n-1}{2}}
  \chi_{A}(\rho) d \rho.
\end{equation*}
  These two integrals have different phases. In $I_{+,A}(r)$ the
  phase $\phi(\rho)=\rho^{a}+r \rho$ satisfies $\phi'(\rho)=a
  \rho^{a-1}+r\ge r$ and we can integrate by parts an arbitrary
  number of times; note that the contribution of the boundary
  terms at $\rho=1$ is bounded uniformly in $r$. Writing $L
  \phi:=-\partial_{\rho}(\frac{\phi}{a \rho^{a-1}+r})$, we have
  for any $N\ge1$
\begin{equation*}
  \textstyle
  I_{+,A}(r)=
  c(a,n,N,r)+
  \int_{1}^{\infty}e^{i (\rho^{a}+r \rho)}
  L^{N}[
  a_{+}(r\rho)\rho^{\frac{n-1}{2}}
  \chi_{A}(\rho)] d \rho
\end{equation*}
  with $|c(a,n,N,r)|\lesssim_{n,N,a}1$. From \eqref{eq:asyapm}
  we have, uniformly in $r\ge1$, $A>0$,
\begin{equation}
  \label{eq:asyaprod}
  |\partial_{\rho}^{k}(a_{\pm}(r\rho)\chi_{A}(\rho)
  \rho^{\frac{n-1}{2}})|
  \lesssim_{\nu,k}\rho^{\frac{n-1}{2}-k}
  \quad\text{for}\quad \rho\ge1,\ k\in \mathbb{N}_{0}.
\end{equation}
  This implies
\begin{equation*}
  \textstyle
  |L^{N}[a_{+}(r\rho)\rho^{\frac{n-1}{2}}
  \chi_{A}(\rho)]|
  \lesssim_{n,N,a}
  r^{-N}\rho^{\frac{n-1}{2}-N}.
\end{equation*}
  Taking e.g.~$N=n+1$ we conclude
\begin{equation*}
  |I_{+,A}(r)|\lesssim_{n,a}1
  \quad\text{for}\quad r\ge1.
\end{equation*}
  Consider next
\begin{equation*}
  \textstyle
  I_{-,A}(r)=\int_{1}^{\infty}e^{i (\rho^{a}-r \rho)}
  a_{-}(r\rho)\rho^{\frac{n-1}{2}}
  \chi_{A}(\rho) d \rho.
\end{equation*}
  Now the phase has a unique, nondegenerate
  stationary point $\rho_{0}$:
\begin{equation*}
  \textstyle
  \phi(\rho)=\rho^{a}-r \rho
  \qquad
  \phi'=a \rho^{a-1}-r=0
  \qquad
  \rho_{0}=(\frac ra)^{\frac{1}{a-1}},
  \qquad
  \phi''(\rho_{0})= a(a-1)(\frac ra)^{\frac{a-2}{a-1}}
\end{equation*}
  Apply the change of variable $\rho= r^{\frac{1}{a-1}}
  u$ to get
\begin{equation*}
  \begin{aligned}
  I_{-,A}(r)
  &=
  r^{\frac{n+1}{2(a-1)}}
  \int_{b}^{\infty}e^{i M(u^{a}-u)}
  a_{-}(Mu)u^{\frac{n-1}{2}}\chi_{A'}(u) d u,\\
  b&=r^{-\frac{1}{a-1}},
  \qquad
  M=r^{\frac{a}{a-1}},
  \qquad
  A'=Ar^{-\frac{1}{a-1}} .
  \end{aligned}
\end{equation*}
  The new phase $\phi_{1}(u)=u^{a}-u$ has a nondegenerate
  critical point at $u_{0}=a^{-\frac{1}{a-1}}<1$, with
  $\phi''_{1}(u_{0})\gtrsim_{a}1$. Note also that $b\in(0,1]$.
  We split the integral as $\int_{b}^{2}+\int_{2}^{\infty}$. For
  the first piece we can use a standard van der Corput estimate:
\begin{equation*}
  \textstyle
  r^{\frac{n+1}{2(a-1)}}|\int_{b}^{2}|\lesssim_{n,a}
  r^{\frac{n+1}{2(a-1)}}M^{-\frac 12}=
  r^{\frac{n+1-a}{2(a-1)}}.
\end{equation*}
  For the $\int_{2}^{\infty}$ piece the phase is non stationary.
  As before, using sufficiently many integrations by parts and
  \eqref{eq:asyaprod}, we get
\begin{equation*}
  \textstyle
  r^{\frac{n+1}{2(a-1)}}|\int_{2}^{\infty}|\lesssim_{n,a}1.
\end{equation*}
  Recalling the external factor $r^{-\frac{n-1}{2}}$ in
  \eqref{eq:decomp} we obtain
\begin{equation*}
  |I_{A}(r)|\lesssim_{n,a}1+r^{-\frac{n-1}{2}}+
  r^{-\frac{n-1}{2}}r^{\frac{n+1-a}{2(a-1)}}\lesssim
  \bra{r}^{\frac{n(2-a)}{2(a-1)}}
\end{equation*}
  and this proves the upper bound.

  It remains only to record the sharpness statement when
  $1<a<2$. This is the standard stationary-phase expansion of
  the full $n$-dimensional oscillatory integral
\begin{equation*}
  \widehat{v}(x)=
  \int_{\mathbb{R}^n}e^{i|\xi|^a-ix\cdot\xi}\,d\xi
\end{equation*}
  up to the harmless normalization convention for the Fourier
  transform. For $x=r\omega$, $|\omega|=1$, the phase
  $|\xi|^a-x\cdot\xi$ has the unique
  nondegenerate critical point
\begin{equation*}
  \xi_0=\left(\frac{r}{a}\right)^{\frac{1}{a-1}}\omega,
  \qquad
  \phi(\xi_0)=
  - (a-1)a^{-\frac{a}{a-1}}r^{\frac{a}{a-1}}.
\end{equation*}
  After the scaling $\xi=r^{\frac{1}{a-1}}\eta$, stationary
  phase near $\xi_0$ gives
\begin{equation*}
  c_{a,n}r^{\frac{n(2-a)}{2(a-1)}}
  e^{-id_ar^{\frac{a}{a-1}}}
\end{equation*}
  with $c_{a,n}\neq0$. The complementary region is nonstationary
  and gives a lower-order contribution. This gives the stated
  asymptotic and concludes the proof.
\end{proof}

\subsection{Compactness of the localized propagator}
\label{sub:comp}

We can now prove a compactness result.

\begin{theorem}[]
\label{the:FracLapComp}
  Let $E,F \subset \mathbb{R}^{n}$ measurable sets of finite
  measure and $H=(-\Delta)^{\alpha}$. Assume that
  \begin{enumerate}[label=(\roman*)]
    \item
    either $\alpha\ge 1$
    \item
    or $\alpha\in(\frac 12,1)$ and \eqref{eq:MeasDecay2} holds.
  \end{enumerate}
  Then $\one{F}e^{-iTH}\one{E}$ is a compact operator on $L^{2}$
  for all $T\neq0$.
\end{theorem}

\begin{proof}
It is enough to prove the compactness of the adjoint operator
$\one{E} e^{i T H} \one{F}$. By formula \eqref{eq:FracSchKer}
and scaling properties of the Fourier transform, $\one{E} e^{i T
H} \one{F}$ can be regarded as an integral operator of kernel
\begin{equation}
  \label{eq:locFracKer}
  \begin{aligned}
  \one{E} e^{i T H} \one{F} (x,y)
  &=
  T^{- \frac{n}{2 \alpha}} \one{E}(x)
  \mathcal{F} (e^{i |\cdot|^{2 \alpha}})
  \bigl(T^{-\frac{1}{2\alpha}} |x-y|\bigr)
  \one{F}(y),
  \end{aligned}
\end{equation}
for $x,y \in \mathbb{R}^{n}$. Applying the estimate of
Proposition \ref{pro:estoscilint} with $a=2\alpha$, the kernel
is bounded when $\alpha\ge1$, and thus square integrable, being supported in the cartesian
product of the finite measure sets $E,F \subset \mathbb{R}^n$. If $\frac12<\alpha<1$, the square of the kernel is
controlled by a constant times
\begin{equation*}
  \one{E}(x)(1+|x-y|^{2\gamma})\one{F}(y),
  \qquad
  \gamma=\frac{n(1-\alpha)}{2\alpha-1},
\end{equation*}
which is integrable by finiteness of the measures of $E$ and $F$ and of the interaction energy $I_{\gamma}(E,F)$, i.e. assumption \eqref{eq:MeasDecay2}.
Thus \eqref{eq:locFracKer} is in $L^2(\mathbb{R}^{2n})$ in both
cases, so $\one{E} e^{i T H} \one{F}$ is Hilbert-Schmidt,
and therefore compact.
\end{proof}

\section{Perturbed fractional one dimensional Schr\"odinger flow}
\label{sec:one_dim_flow}

\subsection{The Jost Functions}
\label{sub:jost_funcs}

We begin the section by recalling the Jost estimates
used below; see
\cite{DeiftTrubowitz79-a, DAnconaFanelli06-a, Mizutani11-a} 
for the underlying one dimensional scattering theory.

If $V \in L^1_N$ for $N \geq 2$ and $0 \leq k \leq N-1$,
\begin{equation}
\label{eq:JostEst1}
| \partial^{k}_{\lambda} f_{\pm}(\lambda,x) | \lesssim \bra{x}^k (1 + \max(\pm x,0)).
\end{equation}
Moreover, if $W(0)=0$, such estimate can be improved by
\begin{equation}
\label{eq:JostEst2}
| \partial^{k}_{\lambda} f_{\pm}(\lambda,x) | \lesssim \bra{x}^k.
\end{equation}

We also recall the equality
\begin{equation}
\label{eq:wronskTransm}
\frac{1}{W(\lambda)} = - \frac{\Theta(\lambda)}{2 i \lambda},
\end{equation}
valid for $V$ at least in $L^1_1$, where $\Theta(\lambda)$ is the so called \textit{transmission} \textit{coefficient}, which is bounded as a function of $\lambda$. 
Moreover, under the additional assumption $V \in L^1_N$, $\Theta(\cdot) \in C^{N-1}(\mathbb{R})$ and one has following bound on its derivatives, 
\begin{equation}
\label{eq:TransmDer}
\partial^{k}_{\lambda} \Theta(\lambda) \lesssim \lambda^{-1}, \qquad 1 \le k \le N-1.
\end{equation}

If $V \in L^1_N(\mathbb{R})$, $W(0) \neq 0$, for $0 \leq k \leq N-1$ one has
\begin{equation}
\label{eq:JostEstGeneric}
| \partial^{k}_{\lambda} (\Theta(\lambda) f_{\pm}(\lambda,x)) | \lesssim \bra{x}^k, \qquad \lambda \neq 0.
\end{equation}

If $V \in L^1_N$, $N \geq 2$, $W(0)=0$, then $\Theta \in C^{N-2}(\mathbb{R})$ and \eqref{eq:TransmDer} holds for $1 \leq k \leq N-2$.

\subsection{Fourier Estimate for the Jost Resolvent Kernel}

Define
\begin{equation}
\label{eq:ArgFourJost}
g_{1,x,y}(\mu) = \one{[0,+\infty)}(\mu) \mu^{\frac{1}{\alpha} - 1} \one{(x,+\infty)}(y) [\frac{f_{+}(\mu^{\frac{1}{2 \alpha}},y) f_{-}(\mu^{\frac{1}{2 \alpha}},x)}{W(\mu^{\frac{1}{2 \alpha}})} ],
\end{equation}
\begin{equation}
\label{eq:ArgFourJost2}
g_{2,x,y}(\mu) = - \one{[0,+\infty)}(\mu) \mu^{\frac{1}{\alpha} - 1} \one{(-\infty,x)}(y) [ \frac{f_{+}(- \mu^{\frac{1}{2 \alpha}},y) f_{-}(-\mu^{\frac{1}{2 \alpha}},x)}{W(-\mu^{\frac{1}{2 \alpha}})}].
\end{equation}
Consider a smooth cut-off $\one{[-1,1]} \le \chi_{1} \le \one{[-2,2]}$, and $\chi_{A} (\mu) = \chi(\mu / A)$.
Define also
\begin{equation}
\label{eq:FourApprox}
g_{i,x,y,A} = \chi_A g_{i,x,y}, \qquad i=1,2.
\end{equation}

\begin{proposition}
\label{prop:FourBoundJost}

For $V$, $N$ as in \textbf{Assumption (A)} (resp. \textbf{Assumption (B)}), $\widehat{g_{i,x,y}}(T)$ is a continuous function for $T \neq 0$ and $\widehat{g_{i,x,y,A}} \rightarrow \widehat{g_{i,x,y}}$ uniformly in $\mathbb{R} \setminus (- \delta, \delta)$, $i =1,2$. Moreover, for any $T \neq 0$
\begin{equation}
\label{eq:FourEstWneq0}
|  \widehat{g_{i,x,y}}(T) | + |  \widehat{g_{i,x,y,A}}(T) | \lesssim_{N,\alpha} (1+\max(\pm x,0))(1+\max(\pm y,0)) (
1+ \frac{\bra{x}^{N-1} \bra{y}^{N-1}}{T^{N-1}}),
\end{equation}
for $V,N$ as in \textbf{Assumption (A)}, $i=1,2$;
\begin{equation}
\label{eq:FourEstWeq0}
| \widehat{g_{i,x,y}}(T) | + |  \widehat{g_{i,x,y,A}}(T) | \lesssim_{N,\alpha} 1 + \frac{\bra{x}^{N-2} \bra{y}^{N-2}}{T^{N-2}},
\end{equation}
for $V,N$ as in \textbf{Assumption (B)}, $i=1,2$.

\end{proposition}

\begin{proof}
We prove the proposition for $g_{1,x,y}$, $g_{1,x,y,A}$. The proof for $g_{2,x,y}$, $g_{2,x,y,A}$ is analogous.
We write
\begin{equation*}
g_{1,x,y} = \chi_1 g_{1,x,y} + (1 - \chi_1) g_{1,x,y}, \qquad g_{1,x,y,A}= \chi_1 g_{1,x,y} + (1 - \chi_1) \chi_A g_{1,x,y}.
\end{equation*}

Let us first focus on the low energy part $\chi_1 g_{1,x,y}$. We claim that it is an $L^1$ function, thus its Fourier transform is continuous and bounded.

Indeed, under \textbf{Assumption (A)}, making use of \eqref{eq:JostEst1} we can estimate:
\begin{align*}
\lVert \widehat{\chi_1 g_{1,x,y}} \rVert_{\infty} 
& \leq \lVert \one{(x,+\infty)}(y) [\frac{f_{+}(\mu^{\frac{1}{2 \alpha}},y) f_{-}(\mu^{\frac{1}{2 \alpha}},x)}{W(\mu^{\frac{1}{2 \alpha}})} ] \rVert_{L_{\mu}^{\infty}([0,2])} \lVert \mu^{\frac{1}{\alpha}-1} \rVert_{L_{\mu}^1([0,2])}
\\ &
\lesssim_{\alpha} (1 + \max(\pm x, 0))(1 + \max(\pm y, 0)).
\end{align*}
Notice that non-vanishing and regularity properties of the Wronskian imply that $\frac{1}{W(\mu^{\frac{1}{2\alpha}})}$ is bounded for $\mu \in [0,2]$.

On the other hand, under assumption \textbf{Assumption (B)}, 
we recall that (see \cite{Mizutani11-a} and 
\cite{DAnconaFanelli06-a})
\begin{equation*}
\mathcal{F}(\frac{\lambda \chi(\lambda)}{W(\lambda))}) \in L^1(\mathbb{R}),
\end{equation*}
for any smooth cut-off function $\chi$. This implies 
\begin{equation*}
\lVert \frac{\mu^{\frac{1}{2 \alpha}}}{W(\mu^{\frac{1}{2 \alpha}})} \rVert_{L^\infty([0,2])} < + \infty,
\end{equation*}
thus we can estimate 
\begin{equation*}
\lVert \widehat{\chi_1 g_{1,x,y}} \rVert_{\infty}
\leq \lVert \one{(x,+\infty)}(y) \mu^{\frac{1}{2 \alpha}} [ \frac{f_{+}(\mu^{\frac{1}{2 \alpha}},y) f_{-}(\mu^{\frac{1}{2 \alpha}},x)}{W(\mu^{\frac{1}{2 \alpha}})} ] \rVert_{L_{\mu}^\infty([0,2])} \lVert \mu^{\frac{1}{2 \alpha}-1} \rVert_{L_{\mu}^1([0,2])}
\lesssim_{\alpha} 1.
\end{equation*}
Now we turn our attention to the high energy parts $(1-\chi_1) g_{1,x,y}$, $(1-\chi_1)\chi_A g_{1,x,y}$. We use the standard property
\begin{equation}
\label{eq:FourTransfDeriv}
i^{M} T^M \widehat{f}(T) = \widehat{f^{(M)}}(T), \qquad M \in \mathbb{N}.
\end{equation}
 
We want to study $[(1- \chi_1) g_{1,x,y}]^{(N-1)}$ and $[(1- \chi_1) g_{1,x,y,A}]^{(N-1)}$ for $N$ as in \textbf{Assumption (A)} (resp. $[(1- \chi_1) g_{1,x,y}]^{(N-2)}$ and $[(1- \chi_1) g_{1,x,y,A}]^{(N-2)}$ for $N$ as in \textbf{Assumption (B)}).
By the Leibniz rule
\begin{equation*}
[(1-\chi_1) g_{1,x,y}]^{(N-1)} = G + R, \qquad [(1-\chi_1) \chi_A g_{1,x,y}]^{(N-1)} = G^A + R^A,
\end{equation*}
where $R$ and $R^A $ group terms with derivatives hitting $(1 - \chi_1$) at least once, and are thus supported in $[1,2]$, while 
\begin{equation}
G = (1 - \chi_1) g_{1,x,y}^{(N-1)}, \qquad G^A = (1 - \chi_1) \sum^{N-1}_{j=0} \chi^{(j)}_A g_{1,x,y}^{(N-1-j)}
\end{equation} 
It is clear that
\begin{equation*}
\lVert R \rVert_{1} + \lVert R^A \rVert_{1} \lesssim_{N,\alpha} (1+\max(\pm x,0))(1+\max(\pm y,0)) \bra{x}^{N-1} \bra{y}^{N-1},
\end{equation*}
uniformly for $A > 2$.

First of all, notice that terms of the form $\chi^{(j)}_A = \chi^{(j)}_{1}(\mu / A) A^{-j}$ enjoy the property $\mu^{j} \chi^{(j)}_A \in L^{\infty}$ uniformly for $A > 2$. 

Let us then study terms of the form $g_{1,x,y}^{(N-1-j)}$, $j = 0, \ldots, N-1$. Making use of \eqref{eq:JostEst1}, \eqref{eq:TransmDer}, one can see that
\begin{align*}
& | [ \Theta(\mu^{\frac{1}{2 \alpha}}) f_{+}(\mu^{\frac{1}{2 \alpha}},y)f_{-}(\mu^{\frac{1}{2 \alpha}},x)) \mu^{\frac{1}{2 \alpha}-1}]^{(N-1-j)} | 
\\ & 
\lesssim_{j,\alpha} \bra{x}^{N-1-j}(1+\max(\pm x,0))) \bra{y}^{N-1-j} (1 + \max (\pm y , 0 ) ) \mu^{(N-1-j)(\frac{1}{2 \alpha} - 1)}.
\end{align*} 
If $N$ is as in \textbf{Assumption (A)}, it is large enough such that
\begin{equation*}
\lVert G \rVert_{1} + \lVert G^A \rVert_{1} \lesssim_{N,\alpha} \bra{x}^{N-1}(1+\max(\pm x,0))) \bra{y}^{N-1} (1 + \max (\pm y , 0 ) ),
\end{equation*}
uniformly for $A > 2$.

Making use of \eqref{eq:FourTransfDeriv}, one has
\begin{align*}
|\widehat{(1-\chi_1 )g_{1,x,y} } (T)| + |\widehat{(1 - \chi_1(\mu))g_{1,x,y,A} } (T)|  
\\
\lesssim_{N,\alpha} \frac{\bra{x}^{N-1}(1+\max(\pm x,0))) \bra{y}^{N-1} (1 + \max (\pm y , 0 ) )}{T^{N-1}}.
\end{align*}

We note that 
$[(1-\chi_1) g_{1,x,y}]^{(N-1)}$ is in $L^{1}$
for $N$ as above, hence 
$[(1-\chi_1)\chi_{A} g_{1,x,y}]^{(N-1)}\to 
  [(1-\chi_1) g_{1,x,y}]^{(N-1)}$
in $L^{1}$ as $A\to \infty$, hence 
$T^{N-1}\mathcal{F}((1-\chi_1) \chi_{A}g_{1,x,y})\to
  T^{N-1}\mathcal{F}((1-\chi_1) g_{1,x,y})$
uniformly on $\mathbb{R}$.
It follows that 
$\mathcal{F}((1-\chi_1) \chi_{A}g_{1,x,y})\to
  \mathcal{F}((1-\chi_1) g_{1,x,y})$
uniformly on $\mathbb{R}\setminus(-\delta,\delta)$ for all
$\delta>0$.
Summing back the continuous function
$\mathcal{F}(\chi_1 g_{1,x,y})$, we conclude that
$\widehat{g_{1,x,y,A}}\to \widehat{g_{1,x,y}}$
uniformly on $\mathbb{R}\setminus(-\delta,\delta)$ for all
$\delta>0$ as $A\to \infty$.

Collecting the estimates for the low and high energy parts, it is clear that \eqref{eq:FourEstWneq0} holds.

One can reason in the same fashion under \textbf{Assumption (B)}, making use of the regularity results for $\Theta$ and $f_{\pm}$ stated in the last part of subsection \ref{sub:jost_funcs} and invoking \eqref{eq:JostEst2}.
\end{proof}

\subsection{Compactness of the localized propagator}
\label{sub:one_dim_comp}

We first record explicitly how the Stone formula enters the
argument. Since we work with $U_{\alpha}(t)$, the pure point
spectrum of $h$ does not contribute to the representation below.

\begin{lemma}[Stone--Jost representation]
\label{lem:StoneJost}
Let $t\neq0$ and let $E,F\subset\mathbb{R}$ be measurable sets.
For $\phi\in C^\infty_c(\mathbb{R})$ one has
\begin{equation}
  \label{eq:StoneJostKernel}
  \one{F}e^{-ith^{\alpha}}(t)\one{E}\phi(x)
  =
  c_\alpha\sum_{j=1}^2
  \int_{\mathbb{R}}K_j(t;x,y)\phi(y)\,dy,
\end{equation}
where
\begin{equation}
  \label{eq:StoneJostKernels}
  K_j(t;x,y)
  =
  \one{F}(x)\one{E}(y)
  \mathcal{F}_{\mu\rightarrow t}
  \bigl(g_{j,x,y}\bigr)(t),
  \qquad j=1,2,
\end{equation}
up to harmless constants depending only on the normalization of
the Fourier transform.
\end{lemma}

\begin{proof}
Since the operator has only absolutely continuous spectrum,
Stone's formula gives
\begin{equation}
  \label{eq:1dStone}
  e^{-ith^{\alpha}}
  =
  \frac{1}{2\pi i}\int_0^\infty
  e^{-it\lambda^\alpha}
  \bigl(R_V(\lambda+i0)-R_V(\lambda-i0)\bigr)\,d\lambda
\end{equation}
in the strong sense. The one dimensional limiting resolvents
have the usual Jost representation; for instance, one of the two
pieces has kernel
\begin{equation}
  \label{eq:resJost1}
  \one{(x,+\infty)}(y)
  \frac{f_+(\sqrt{\lambda},y)f_-(\sqrt{\lambda},x)}
  {W(\sqrt{\lambda})},
\end{equation}
and the conjugate boundary value gives the analogous term with
$-\sqrt{\lambda}$ and $\one{(-\infty,x)}(y)$. After the change
of variables $\mu=\lambda^\alpha$, these two pieces are
precisely the functions $g_{1,x,y}$ and
$g_{2,x,y}$ defined in
\eqref{eq:ArgFourJost}--\eqref{eq:ArgFourJost2}. To justify the
exchange of the $\mu$ and $y$ integrals, insert first a factor
$e^{-\varepsilon\mu}$ in \eqref{eq:1dStone}; the resulting
integrals are absolutely convergent. The estimates in
Proposition \ref{prop:FourBoundJost}, followed by dominated
convergence as $\varepsilon\downarrow0$,
give \eqref{eq:StoneJostKernel}.
\end{proof}

\begin{theorem}
\label{the:1dComp}
  Let $V$ satisfy Assumption \textbf{(A)}. For any $T \neq 0$
  and any finite measure sets $E,F\subset \mathbb{R}$ such that
  \begin{equation}
    \label{eq:MeasDecayJostA}
    \int_{\mathbb{R} \times \mathbb{R}}
    \one{F}(x) \bra{x}^{N-1}(1+\max(\pm x,0)) \bra{y}^{N-1} (1 + \max (\pm y , 0 ))
    \one{E}(y) \, dx dy
    < + \infty,
  \end{equation}
  the operator $\one{F}e^{-iTh^{\alpha}}\one{E}$ is Hilbert--Schmidt,
  and hence compact.

  The same conclusion holds if $V$ satisfies Assumption
  \textbf{(B)}, for any $T \neq 0$ and any finite measure 
  sets $E,F \subset \mathbb{R}$ such that
  \begin{equation}
    \label{eq:MeasDecayJostB}
    \int_{\mathbb{R} \times \mathbb{R}}
    \one{F}(x) \bra{x}^{N-2} \bra{y}^{N-2}
    \one{E}(y) \, dx dy
    < + \infty.
  \end{equation}
\end{theorem}

\begin{proof}
By Lemma \ref{lem:StoneJost}, the localized propagator has
integral kernel given by the sum of the two kernels in
\eqref{eq:StoneJostKernels}. Under Assumption \textbf{(A)}, \eqref{eq:FourEstWneq0} implies that this kernel belongs to
$L^2(\mathbb{R}\times\mathbb{R})$ whenever
\eqref{eq:MeasDecayJostA} holds. Under Assumption \textbf{(B)},
the same conclusion follows from \eqref{eq:FourEstWeq0}, and \eqref{eq:MeasDecayJostB}. Hence
$\one{F}e^{-iTh^{\alpha}}\one{E}$ is Hilbert--Schmidt in both cases.
\end{proof}

\section{Proof of the main results}
\label{sec:proomainres}

\begin{proof}[Proof of Theorem \ref{the:FracDynAB}]
Let us set $H = (- \Delta )^{\alpha}$,
\begin{equation}
  \label{eq:opDef}
  S = e^{i T H} \one{F} e^{-i T H} \one{E}
\end{equation}
By Proposition \ref{prop:compactCriterion}, inequality
\eqref{eq:AmBerIneq} is a consequence of
\begin{equation}
  \label{eq:opNormIneq}
  \lVert S \rVert < 1.
\end{equation}
In turn, since the left-hand term in \eqref{eq:opNormIneq} is
trivially $\leq 1$, the statement is proved by contradiction,
supposing that equality holds in \eqref{eq:opNormIneq}.

First, we can apply Theorem \ref{the:FracLapComp} to deduce that
$S$ is a compact operator. By Proposition
\ref{prop:compactCriterion}, compactness yields the existence of
a $1$-eigenfunction $f \in L^2$ for $S$.

Now, define inductively $E_0 = E$, $E_k = E_{k-1} \cup (E_{k-1} + z_{k-1}) $ for $1 \leq k \in \mathbb{N}$. The sequence $z_k$ is $\mathbb{R}^n$-valued, and can be chosen so that the following properties hold:
\begin{enumerate}
\item  \textbf{Property} \textbf{(1)}: one has that
\begin{equation*}
|E_{k-1}| \leq |E_k| \leq |E_{k-1}| + 2^{-k}.  
\end{equation*}
In particular, the set $E_{\infty} = \bigcup_{k=0}^{\infty} E_{k}$ still enjoys finiteness of the Lebesgue measure.
\item \textbf{Property} \textbf{(2)}: if $E$, $F$ satisfy assumption \eqref{eq:MeasDecay2}, then \eqref{eq:MeasDecay2} holds also replacing $E$ with $E_{\infty}$.
\end{enumerate}

We show how to choose an appropriate sequence $z_k$. First, define 
\begin{equation}
\label{eq:intDef}
I_k = \int_{\mathbb{R}^n \times \mathbb{R}^n} \one{F}(x)|x-y|^{2 \gamma} \one{E_k}(y) dx dy,
\end{equation}
\begin{align}
R_k = & \int_{\mathbb{R}^n \times \mathbb{R}^n} \one{F}(x)|x-y|^{2 \gamma} \one{E_{k-1}+z_{k-1} }(y) dx dy \label{eq:intRemaind1} \\
& - \int_{\mathbb{R}^n \times \mathbb{R}^n} \one{F}(x)|x-y|^{2 \gamma} \one{(E_{k-1}+z_{k-1} ) \cap E_{k-1}}(y) dx dy. \label{eq:intRemaind2}
\end{align}
We proceed inductively. $I_0$ is finite by assumption \eqref{eq:MeasDecay2}. 
When $k \geq 1$, make the inductive assumption that $I_{k-1}$ is finite. Clearly
\begin{equation}
\label{eq:intSum}
I_k = I_{k-1} +  R_k.
\end{equation}
By Lebesgue dominated convergence theorem, both integrals in \eqref{eq:intRemaind1} and \eqref{eq:intRemaind2} converge to $I_{k-1}$ when $z_{k-1}$ approaches $0$. Thus, choosing $z_{k-1}$ close enough to zero,
\begin{align*}
\label{eq:intRemBound}
|R_{k}| \leq 2^{-k}.
\end{align*}
Up to choosing $z_{k-1}$ smaller, $L^1$ continuity of translations also ensures us that \textbf{Property} \textbf{(1)} is satisfied.
\textbf{Property} \textbf{(2)} now follows from
\begin{equation*}
I_k \leq I_0 + \sum_{j=0}^{k-1}2^{-j},
\end{equation*}
letting $k \rightarrow \infty$, by monotone convergence. 

By a similar argument, we can extract a subsequence of $z_k$ (which for simplicity we still denote $z_k$) and repeat our construction of sets $F_k$, $F_{\infty}$ such that \textbf{Property} \textbf{(1)} holds substituting $E_k$ (resp. $E_{k-1}$, $E_{\infty}$) with $F_k$ (resp. $F_{k-1}$, $F_{\infty}$), and \eqref{eq:MeasDecay2} remains true replacing $E$ with $E_{\infty}$ and $F$ with $F_{\infty}$.

Now, replacing $E$ and $F$ with $E_{\infty}$ and $F_{\infty}$ also in the definition of the operator \eqref{eq:opDef}, we still get a compact operator.
Since by construction each element of the family of linearly independent functions $\{ f(\cdot - z_k) \}_{k \geq 1}$  is a $1$-eigenfunction for this new operator,
we reach the contradiction.
\end{proof}

\begin{proof}[Proof of Theorem \ref{the:PotFracDynAB}]
Set
\begin{equation*}
  S=\one{F}e^{-iTh^{\alpha}}\one{E}.
\end{equation*}
As in the proof of Theorem \ref{the:FracDynAB}, the estimate
\eqref{eq:AmBerIneq} follows once we know that $\lVert
S\rVert<1$ on $L^2(\mathbb{R})$. Since $E$ and $F$ are compact,
they satisfy \eqref{eq:MeasDecayJostA} and
\eqref{eq:MeasDecayJostB}; hence Theorem \ref{the:1dComp} gives
compactness of $S$ under either threshold assumption.

Assume by contradiction that $\lVert S\rVert=1$. Compactness
gives a maximizer $f\neq0$. Equality in the chain of
contractions defining $S$ then implies
\begin{equation*}
  f\in L^2(\mathbb{R}), \qquad
  \operatorname{spt} f\subset E, \qquad
  \operatorname{spt} e^{-iTh^{\alpha}}f\subset F.
\end{equation*}
At this point one repeats the localization argument of
\cite{DAnconaFiorletta26-a}: applying a short family of
finite-propagation wave operators associated with $h+c_0$, for
$c_0$ sufficiently large, to the extremizer, one constructs a
nonzero compactly supported solution of
\begin{equation*}
  (-\partial_x^2+V)g=cg
\end{equation*}
for some real $c$. This is impossible by one dimensional unique
continuation for the second order equation
with $V\in L^1_{\operatorname{loc}}$,
see e.g.~Theorem 7.1 in \cite{SchechterSimon80-a}.
\end{proof}

\begin{proof}[Proof of Theorem \ref{the:hiOrdAB}]
Let us set $H = (-\Delta)^m + V$, and define $S$ as in
\eqref{eq:opDef}. The proof follows the same lines as the
previous ones.

Thus we start by showing that the operator
$S$ is compact. We recall that, under our assumptions,
the spectral decomposition
\begin{equation*}
  L^2 = L^2_p + L^2_{ac}
\end{equation*}
holds, where $L^2_p$ is the subspace of $L^2$ associated to the
pure point part of the spectrum, while $L^2_{ac}$ the one
associated to the absolutely continuous part of the spectrum. We
denote by $P_p$ the projection operator over $L^2_p$, $P_{ac}$
the projection operator over $L^2_{ac}$. We can now split the
localized propagator in the sum of two terms.
\begin{equation}
  \label{eq:locPropSplit}
  \one{F} e^{- i T H} \one{E}
  =
  \one{F} e^{- i T H} P_{p} \one{E}
  +
  \one{F} e^{- i T H} P_{ac} \one{E}.
\end{equation}
We now show that each of the two terms in the right hand side of
\eqref{eq:locPropSplit} is compact, in order to conclude that
their sum is also compact, by standard properties of compact
operators. This in turn yields the compactness of $S$.

Noticing that $P_{p}$ maps $L^2(\mathbb{R}^n)$ into the Sobolev space $H^{2m}(\mathbb{R}^n)$ boundedly 
(see \cite{ConstantinSaut89-a})
and that $H^{2m}(\mathbb{R}^n)$ is compactly embedded in $L^2(\mathbb{R}^n)$, $P_{p}$ is compact as an operator acting on $L^2$.  
Thus the first term in \eqref{eq:locPropSplit} is compact.

For the second one, we will exploit the fact that $S$ factorizes through $L^{\infty}$. Indeed, by \cite{ErdoganGreen23-a}, the following dispersive estimate holds:
\begin{equation}
\label{eq:dispEstSch}
\lVert e^{- i T H} P_{ac} \rVert_{1 \rightarrow \infty} \lesssim T^{-\frac{n}{2m}}.
\end{equation}
Therefore $\one{F} e^{- i T H} P_{ac} \one{E}$ maps $L^2(\mathbb{R}^n)$ into $L^2(F) \cap L^{\infty}$ continuously. We can thus invoke Theorem 4.6.1 in 
\cite{Arendt06-a} with $X=L^2(\mathbb{R}^n)$, $q=2$, $\Omega = F$,
and the trivial fact that $L^2(F)$ is continuously embedded in $L^2(\mathbb{R}^n)$, to conclude that $\one{F} e^{- i T H} P_{ac} \one{E}$ is compact.

As before, we now assume by contradiction that $\|S\|=1$.
The usual abstract procedure allows us to construct a nonzero
$1$-eigenfunction $f$ for the operator $S$, i.e. $S f = f$.
Notice that $f$ is a $1$-eigenfunction of $S$ if and only if
$\operatorname{supp} f \subset E$, $\operatorname{supp} e^{- i T
H} f \subset F$.

We now want to apply powers of the (unbounded) operator $H$ to
$f$, and show that each $H^k f$, for $k \in \mathbb{N}$, is
in $L^{2}$ and is again a $1$-eigenfunction for $S$. 
We shall prove this for $k=1$; the cases $k>1$ follows easily
by induction.

Consider the operator
\begin{equation*}
  He^{-iTH}= He^{-iTH}P_{p}+ He^{-iTH}P_{ac}.
\end{equation*}
One has obviously 
\begin{equation*}
  He^{-iTH}P_{p}:L^{2}\to L^{2}.
\end{equation*}
For the second term, we recall that under our assumptions 
the wave operators $W_{\pm}$ intertwining
$HP_{ac}$ with $(-\Delta)^{m}$ are well defined and are bounded
on $L^{1}$ and on $L^{\infty}$ by the results of
\cite{ErdoganGreen23-a}. Now write
\begin{equation*}
  H e^{- i T H} P_{ac}
  =
  W_{+} (- \Delta)^m e^{- i T (- \Delta)^m}
  W_{+}^{*}
\end{equation*}
and note that, by \cite{Balabane89-a} with 
$P(D) = |D|^{2m}$, $r = 2m$, the operator
$(- \Delta)^m e^{- i T (- \Delta)^m}$ extends to a bounded
operator from $L^{1}$ to $L^{\infty}$. This implies
\begin{equation*}
  He^{-iTH}P_{ac}:L^{1}\to L^{\infty}.
\end{equation*}
Summing up, we have proved that
\begin{equation*}
  He^{-iTH}:L^{1}\cap L^{2}\to L^{\infty}+L^{2}
\end{equation*}
is a bounded operator. 

Now, since $\spt f \subset E$, we have $f\in L^{1}\cap L^{2}$
and hence $He^{-iTH}f\in L^{2}+L^{\infty}$.
Moreover, $e^{-iTH}f$ is supported in $F$, hence
the same is true of $He^{-iTH}f$, which implies that
$He^{-iTH}f=e^{-iTH}Hf\in L^{2}$ with support in $F$.
Therefore, $Hf\in L^{2}$ and since $f$ has support in $E$
the same holds for $Hf$. Summing up, we have proved that
$Hf\in L^{2}$, $\spt Hf \subset E$, $\spt e^{-iTH}Hf \subset F$,
so that $Hf$ is also 1-eigenfunction of $S$, as claimed.

By induction it follows that
$\{ H^k f \}_{k \in \mathbb{N}}$ is a family of
$1$-eigenfunctions for $S$. Since $S$ is compact, we can find
a finite linear
combinations of such functions equal to zero. In other words,
$\exists N \in \mathbb{N}$ and $P$ a polynomial of degree
$N$ such that
\begin{equation*}
  P(H) f = 0
\end{equation*}
Thus, we can find $0 \neq g \in L^2$, $\operatorname{supp} g
\subset E$, $\operatorname{supp} e^{- i T H} g \subset F$, $c
\in \mathbb{R}$ such that $H g = c g$. 

This leads to a contradiction, thanks to known unique continuation properties. For instance, we can apply Theorem 3.1 in \cite{Laba88-a}, for $\mu = m$, $\Omega = \mathbb{R}^n$, noticing that our potential is in $L_{loc}^{n / 2m}(\Omega)$ thanks to \textbf{Assumption (H)}, and that $g$ is in $H^{2m,2}(\Omega) \subset H_{loc}^{2m,q}(\Omega)$ (where $q = 2n / (n + 2m ) < 2$) by construction.
\end{proof}

\printbibliography



%


\end{document}